\documentclass[11pt]{amsart}
\usepackage[margin=1in]{geometry}
\usepackage{amscd,amssymb, amsmath, wasysym}
\usepackage{graphicx}
\usepackage{amsfonts}
\usepackage{mathrsfs}    
\usepackage{amsmath}    
\usepackage{amsthm}     
\usepackage{amscd}      
\usepackage{amssymb}    
\usepackage{eucal}      
\usepackage{latexsym}   
\usepackage{graphicx}   
\usepackage{verbatim}   
\usepackage[all]{xy}     

\makeatletter

\newcounter{thmcounter}

\numberwithin{thmcounter}{section}
\numberwithin{equation}{thmcounter}

\newtheorem{theorem}[thmcounter]{Theorem}
\newtheorem{proposition}[thmcounter]{Proposition}

\newtheorem{corollary}[thmcounter]{Corollary}

\theoremstyle{definition}
\newtheorem{definition}[thmcounter]{Definition}

\newtheorem{remark}[thmcounter]{Remark}

\newtheoremstyle{claim}{9pt}{3pt}{}{\parindent}{\bf}{.}{1em}{}

\theoremstyle{claim}
\newtheorem{claim}[equation]{Claim}

\newenvironment{namelist}[1]{%
\begin{list}{}
{
\settowidth{\labelwidth}{#1}%
\setlength{\labelsep}{0.3em}%
\setlength{\leftmargin}{\labelwidth}%
\addtolength{\leftmargin}{\labelsep}}}{%
\end{list}}

\newcommand{\nZ}{\mathbb{Z}}                     

\newcommand{\nP}{\mathbf{P}}                     

\newcommand{\sO}{\mathscr{O}}                    

\newcommand{\sK}{\mathscr{K}}

\DeclareMathOperator{\gon}{gon}                  

\DeclareMathOperator{\Sym}{Sym}                  
\DeclareMathOperator{\sHom}{\mathscr{H}om}       

\DeclareMathOperator{\rank}{rank}                

\newcounter{rkcounter}             
\begin{document}

\title[A note on the gonality conjecture]{A note on an effective bound for the gonality conjecture}
\author{Alexander Duncan, Wenbo Niu, Jinhyung Park}

\address{{Department of Mathematical Sciences, University of Arkansas, Fayetteville, AR 72701, USA}}
\email{{asduncan@uark.edu}}

\address{Department of Mathematical Sciences, University of Arkansas, Fayetteville, AR 72701, USA}
\email{wenboniu@uark.edu}

\address{Department of Mathematical Sciences, KAIST, 291 Daehak-ro, Yuseong-gu, Daejeon 34141, Republic of Korea}
\email{parkjh13@kaist.ac.kr}

\date{\today}


\thanks{W. Niu was supported by the Simons Collaboration Grants for Mathematicians}

\thanks{J. Park was partially supported by the National Research Foundation (NRF) funded by the Korea government (MSIT) (NRF-2021R1C1C1005479).}

\begin{abstract} 
The gonality conjecture, proved by Ein--Lazarsfeld, asserts that the gonality of a nonsingular projective curve  of genus $g$ can be detected from its syzygies in the embedding given by a line bundle of sufficiently large degree. An effective result obtained by Rathmann says that any line bundle of degree at least $4g-3$ would work in the gonality theorem. In this note, we improve the degree bound to $4g-4$ with two exceptional cases. 
\end{abstract}

\maketitle

\section{Introduction}
\noindent We work over the field $\mathbf{C}$ of complex numbers. Let $C$ be a nonsingular projective curve of genus $g$, and $B$ and $L$ be line bundles on $C$. Suppose that $L$ is globally generated, and write $S=\Sym H^0(L)$. The associated section module 
$$
R(C, B;L):=\bigoplus_{q\geq 0}H^0(B\otimes L^q)
$$
is a finitely generated graded $S$-module. It admits a minimal graded free resolution over $S$:
$$
\cdots\longrightarrow E_p\longrightarrow \cdots \longrightarrow E\longrightarrow E_1\longrightarrow E_0\longrightarrow R(C, B;L)\longrightarrow 0.
$$
Each graded free $S$-module $E_p$ in the resolution has the form 
$$
E_p=\bigoplus_{q\in\nZ}K_{p,q}(C,B;L)\otimes S(-p-q),
$$
where $K_{p,q}(C,B;L)$ is the Koszul cohomology group defined as the cohomology at the middle of the Koszul-type complex
$$
\wedge^{p+1}H^0(L)\otimes H^0(B\otimes L^{q+1})\longrightarrow\wedge^pH^0(L)\otimes H^0(B\otimes L^q)\longrightarrow \wedge^{p-1}H^0(L)\otimes H^0(B\otimes L^{q-1}).
$$

A particularly interesting case is when $B=\sO_C$ and $L$ has large degree. The Koszul groups $K_{p,q}(C;L):=K_{p,q}(C,\sO_C;L)$ give a minimal free resolution of the section ring $R(C; L):=R(C, \sO_C; L)$. Note that if $C \subseteq \nP(H^0(L))$ is projectively normal, then $R(C; L)$ is the homogeneous coordinate ring of $C$ in $\nP(H^0(L))$.
If $L$ is nonspecial, then $K_{p,q}(C;L)$ vanishes for $q\geq 3$ and thus the minimal free resolution essentially consists of two strands $K_{p,1}(C;L)$ and $K_{p,2}(C;L)$. If $L$ has degree larger than $2g$, Green's $(2g+1+p)$ theorem (\cite[Theorem 4.a.1]{Green:KoszulI}, see also \cite{Lazarsfeld:SyzFiniteSet}) determines the strand $K_{p,2}(C; L)$ using the notion of $N_p$ property. The Green--Lazarsfeld gonality conjecture \cite{Lazarsfeld:ProjNormCurve} together with their nonvanishing theorem \cite[Appendix]{Green:KoszulI} predicts the shape of the strand $K_{p,1}(C;L)$. This conjecture was proved by Ein-Lazarsfeld \cite{Ein:Gonality} when $L$ has sufficiently large degree.  Ranthmann then showed an effective result (see \cite[Theorem 1.1]{RATH}): if $\deg L\geq 4g-3$, then 
$$
K_{p,1}(C;L)\neq 0 \Longleftrightarrow 1\leq p\leq \deg L-g-\gon(C),
$$
where $\gon(C)$ is the gonality of $C$ which by definition is the minimal degree of pencils on $C$.
As pointed out in \cite{RATH}, although the degree bound above is not expected to be optimal, there is an example of a plane quartic curve showing the degree bound $4g-4$ does not work. In this short note, we investigate the failure of the gonality conjecture when $\deg L = 4g-4$. The main result is the following.

\begin{theorem}\label{main01} 
Let $C$ be a nonsingular projective curve of genus $g\geq 2$, and $L$ be a line bundle on $C$ with $\deg L\geq 4g-4$. Then 
$$
K_{p,1}(C;L)\neq 0\Longleftrightarrow 1\leq p\leq \deg L-g-\gon(C)
$$
unless $L=\omega^2_C$ and either $g=2$ or $C$ is a plane quartic curve. In the exceptional cases, $K_{\deg L -g-\gon(C)+1, 1}(C; L) \neq 0$ but $K_{\deg L -g-\gon(C)+2, 1}(C; L)= 0$.
\end{theorem}

An easy application of the theorem gives a uniform picture of syzygies of pluricanonical embedding of curves, especially the second power of the canonical divisor.  It has been a long standing interest to understand the syzygies of canonical curves. The shape of the minimal free resolution of $R(C; \omega_C)$ was predicted in Green's conjecture \cite[Conjecture 5.1]{Green:KoszulI}. It was verified by Voisin \cite{Voisin1, Voisin2} for general curves, but it is still widely open in general. For pluricanonical embedding $C \subseteq \nP(H^0(\omega_C^k))$, the picture of syzygies turns out to be complete, and we give a summary here.
Let $C$ be a curve of genus $g\geq 2$ and gonality $\gon(C)$. Put $L:=\omega^k_C$ and write $r:=h^0(L)-1$. For $k\geq 3$, Greens's $(2g+1+p)$-theorem and Rathmann's effective gonality theorem give the result that 
$$
K_{p,1}(C;\omega^k_C)\neq 0\Longleftrightarrow 1\leq p\leq r-\gon(C).
$$
For  $L=\omega^2_C$ and $r=3g-3$, Green's $(2g+1+p)$-theorem and Theorem \ref{main02} gives us the following two cases
\begin{enumerate}
	\item If either $g=2$ ($r=\gon(C)=2$), or  $C$ is a plane quartic curve ($r=5$ and $\gon(C)=3$), then
		$$
		K_{p,1}(C;\omega^2_C)\neq 0\Longleftrightarrow 1\leq p\leq r-\gon(C)+1.
		$$
	\item Otherwise,
		$$
		K_{p,1}(C;\omega^2_C)\neq 0\Longleftrightarrow 1\leq p\leq r-\gon(C).
		$$
\end{enumerate}

In the setting of Theorem \ref{main01}, Green--Lazarsfeld's nonvanishing theorem \cite[Appendix]{Green:KoszulI} shows that $K_{p,1}(C;L)\neq 0$ for $1 \leq p \leq \deg L - g-\gon(C)$. To prove the theorem, it is sufficient to prove that $K_{\deg L -g-\gon(C)+1, 1}(C; L) =0$.
By the duality theorem \cite[Theorem 2.c.6]{Green:KoszulI}, 
$$
K_{\deg L -g-\gon(C)+1, 1}(C; L) = K_{\gon(C)-2, 1}(C, \omega_C; L)^{\vee}.
$$
Notice that $\omega_C$ is $(\gon(C)-2)$-very ample. Recall that $B$ is $p$-very ample if the restriction map on global sections $H^0(B)\rightarrow H^0(B|_{\xi})$ is surjective for every effective divisor $\xi$ of degree $p+1$, (in other words, $\xi$ imposes independent conditions on the global sections of $B$). 
As in \cite{Ein:Gonality} and \cite{RATH}, it is natural to study more generally vanishing of $K_{p,1}(C, B; L)$ when $B$ is a $p$-very ample line bundle and $\deg L \geq \deg B + 2g-2$. The main result of \cite{RATH} says that if $H^1(C, L \otimes B^{-1})=0$, then $K_{p,1}(C, B; L)=0$. For our purpose, we only need to consider the case that $L=B \otimes \omega_C$. Theorem \ref{main01} can be deduced from the following:

\begin{theorem}\label{main02} 
Let $C$ be a nonsingular projective curve of genus $g \geq 0$, $B$ be a $p$-very ample line bundle on $C$, and $L:=B\otimes \omega_C$. 
	\begin{enumerate}
		\item If $h^0(B)\geq p+3$, then $K_{p,1}(C,B;L)=0$.
		\item If $h^0(B)=p+2$, then $K_{p,1}(C,B;L)=S^pH^0(\omega_C)$.
		\item If $h^0(B)=p+1$, then $K_{p,1}(C,B;L)=0$.
	\end{enumerate}
\end{theorem}

The idea to prove the theorem is to use the kernel bundles on the symmetric products of the curve. We follow the approach introduced by Voisin \cite{Voisin1, Voisin2} and then used by Ein--Lazarsfeld \cite{Ein:Gonality}, Rathmann \cite{RATH}, and many others to conduct a computation of Koszul cohomology groups on the symmetric products of the curve. To be concrete, in our case,
$$
K_{p,1}(C, B; L) = H^1(C_{p+1}, M_{p+1,B} \otimes N_{p+1,L}),
$$
where $M_{p+1,B}$ is the kernel bundle of the evaluation map $H^0(C, B) \otimes \sO_{C_{p+1}} \to E_{p+1, B}$ of the tautological bundle $E_{p+1, B}$ and $N_{p+1,L}$ is a line bundle on $C_{p+1}$. More generally, we establish the following vanishing:
$$
H^i(C_{p+1}, \wedge^k M_{p+1, B} \otimes N_{p+1, L}) = 0~~\text{ for $i>0$}
$$
when $h^0(B) \geq p+k+2$. We hope that our results and methods may shed lights on the similar problems for higher dimensional varieties.

\bigskip
\noindent {{\bf \em Acknowledgments.} The authors would like to thank Lawrence Ein for suggestions and comments.

\section{Preliminaries}\label{sec:prelim}

\noindent Let us start with setting up notations used throughout the paper. Let $C$ be a nonsingular projective curve of genus $g$. For any $p\geq 0$, denote by $C_{p+1}$ the $(p+1)$-th symmetric product. Write $U_{p+1}=C_p\times C$ to be the universal family over $C_{p+1}$. One has a commutative diagram 
$$\xymatrix{
	U_{p+1}=C_{p}\times C \ \ar[dr]_{\sigma_{p+1}} \ar@{^{(}->}[r]^-j &C_{p+1}\times C\ar[d]^{\pi_{p+1}}\\
	&C_{p+1}}$$
in which $\pi_{p+1}$ is the projection map, $j$ is an embedding defined by $j(\xi, x)=(\xi+x, x)$, and $\sigma_{p+1}=\pi_{p+1}|_{U_{p+1}}$ so that $\sigma_{p+1}(\xi, x)=\xi+x$. Write $pr \colon U_{p+1}\rightarrow C$ to be the projection map to $C$.

\begin{definition}
	Let $B$ be a line bundle on  $C$. For $p\geq 0$, define 
	$$E_{p+1,B}=\sigma_{p+1,*}(pr^*B) \text{ and }N_{p+1,B}=\det E_{p+1,B}.$$
\end{definition}

\begin{remark}
	For basic properties of the vector bundles $E_{p+1,B}$ and the line bundle $N_{p+1,B}$, we refer the reader to the paper \cite{ENP}. Here we mention that $N_{p+1,B}=S_{p+1,B}(-\delta_{p+1})$, where $S_{p+1,B}$ is the invariant descend of 
$$
B^{\boxtimes p+1}=\underbrace{B\boxtimes  \cdots \boxtimes B}_{\text{$p+1$ times}}
$$ 
on $C^{p+1}$ to $C_{p+1}$ under the action of the permutations group $\mathfrak{S}_{p+1}$ on $C^{p+1}$ and $\sO_{C_{p+1}}(-\delta_{p+1})=N_{p+1,\sO_C}$.
\end{remark}

Let $B$ be a $p$-very ample line bundle on $C$. As the fiber of $E_{p+1, B}$ over $\xi \in C_{p+1}$ is $H^0(B|_{\xi})$, the evaluation map $H^0(B)\otimes \sO_{C_{p+1}}\rightarrow E_{p+1,B}$ on global sections is surjective. Define $M_{p+1,B}$ to be the kernel bundle of the evaluation map. We obtain a short exact sequence 
$$0\longrightarrow M_{p+1,B}\longrightarrow H^0(B)\otimes \sO_{C_{p+1}}\longrightarrow E_{p+1,B}\longrightarrow 0.$$ 
The following vanishing theorem about the kernel bundle $M_{p+1,B}$ is an immediate consequence of Ranthmann's vanishing theorem on Cartesian products of the curve.

\begin{proposition}\label{lm:02} Let $B$ be a $p$-very ample line bundle on $C$, and $L$ be a globally generated line bundle on $C$ such that $h^1(L)=h^1(L\otimes B^{-1})=0$. Then one has 
	$$H^k(C_{p+1},\wedge^mM_{p+1,B}\otimes N_{p+1,L})=0, \text{ for all }k>0,m>0.$$
\end{proposition}

\begin{proof} By \cite[Theorem 3.1]{RATH}, one has the vanishing 
	$$H^k(C^{p+1},q^*(\wedge^m M_{p+1,B}\otimes N_{p+1,L}))=0 \text{ for all }k>0, m>0,$$
	where $q \colon C^{p+1}\rightarrow C_{p+1}$ is the quotient map. Since $q$ is finite, $\sO_{C_{p+1}}$ is a direct summand of $q_*\sO_{C^{p+1}}$. Thus by projection formula, $\wedge^m M_{p+1,B}\otimes N_{p+1,L}$ is a direct summand of $q_*(q^*(\wedge^m M_{p+1,B}\otimes N_{p+1,L}))$, from which the result follows.  
\end{proof}

Next we prove a crucial property of kernel bundle $M_{p+1,B}$, which is important for us to use the inductive argument.

\begin{proposition}\label{lm:01} Let $B$ be a $p$-very ample line bundle. There is a short exact sequence 
	$$0\longrightarrow \sigma^*_{p+1}M_{p+1,B}\longrightarrow M_{p,B}\boxtimes \sO_C\longrightarrow (\sO_{C_p}\boxtimes B) (-U_p)\longrightarrow 0.$$
\end{proposition}

\begin{proof} Denote by $\alpha \colon M_{p,B}\boxtimes \sO_C\rightarrow (\sO_{C_p}\boxtimes B) (-U_p)$ the morphism appeared on the right hand side of the sequence. We first show that it is surjective. Indeed, choose any $\xi\in C_p$, and consider the fiber $C=\{\xi\}\times C\subseteq C_p\times C$ over $\xi$. Restricting $\alpha$ to this fiber yields the evaluation map 
$$\alpha_{\xi}:H^0(B(-\xi))\otimes \sO_C\longrightarrow B(-\xi).$$
Since $B$ is $p$-very ample and $\xi$ has degree $p$, it follows that $B(-\xi)$ is $0$-very ample and thus globally generated. Hence $\alpha_{\xi}$ is surjective. This means that  $\alpha$ is surjective. 

Next we consider the following fiber product diagram
$$
\begin{CD}
	@.C_p\times C\times C @>\bar{\sigma}>> C_{p+1}\times C\  @.\supseteq U_{p+1} \\
	@.@V\bar{\pi}VV @VV\pi_{p+1}V \\
	U_{p+1}=@.C_{p}\times C @>>\sigma_{p+1}> C_{p+1}.
\end{CD}
$$
On $C_p\times C\times C$, we have two divisors $D_0$ and $D_1$ defined in the way that $D_0$ is the image of 
$$
C_p\times C \longrightarrow C_p\times C\times C, \quad  (\xi,x)\longmapsto(\xi, x,x),
$$
and $D_1$ is the image of 
$$C_{p-1}\times C\times C\longrightarrow C_p\times C\times C, \quad (\xi, y,x)\longmapsto (\xi+x,y,x).
$$ 
Observe that 
$$
\bar{\sigma}^*U_{p+1}=D_0+D_1 \text{ and }D_0\cap D_1=C_{p-1}\times C.
$$
It is easy to check that 
$$
\sigma^*_{p+1}M_{p+1,B}=\bar{\pi}_*(pr^*B(-D_0-D_1) \text{ and } M_{p,B}\boxtimes\sO=\bar{\pi}_*pr^*B(-D_1),
$$
where $pr:C_p\times C\times C\rightarrow C$ is the projection to the right hand side component $C$.
Now we can form a short exact sequence on  $C_p\times C\times C$,
$$0\longrightarrow \sO(-D_0-D_1)\longrightarrow \sO(-D_1)\longrightarrow \sO_{D_0}(-D_1)\longrightarrow 0.$$
Note that $\sO_{D_0}(-D_1)= \sO_{C_p\times C}(-U_p)$. Tensoring the short exact sequence with $pr^*B$ and then pushing it down to $C_p\times C$, we obtain the desired short exact sequence.
\end{proof}

\begin{remark} 
The proof above shows that for any line bundle $B$ (not necessarily $p$-very ample), one has a short exact sequence 
$$
0\longrightarrow pr^*B(-U_{p})\longrightarrow \sigma^*_{p+1}E_{p+1,B}\longrightarrow E_{p,B}\boxtimes \sO_C\longrightarrow 0
$$
on the universal family $U_{p+1}$.
\end{remark}

\section{Proofs of Main Results}

\noindent In this section, we prove the main results of the paper -- Theorems \ref{main01} and \ref{main02}. We keep using the notations introduced in Section \ref{sec:prelim}. On the universal family $U_{p+1}$, consider the short exact sequence 
\begin{equation}\label{eq:10}
	0\longrightarrow \sO_{U_{p+1}}\longrightarrow \sO_{U_{p+1}}(U_p)\longrightarrow \sO_{U_p}(U_p)\longrightarrow 0
\end{equation}
associated to the divisor $U_{p}$. The normal sheaf $\sO_{U_{p}}(U_p)$ of $U_p$ in $U_{p+1}$ can be expressed as 
$$
\sO_{U_{p}}(U_p)\cong (\sO_{C_{p-1}}\boxtimes \omega^{-1}_C)(U_{p-1}).
$$
Let $L$ be a line bundle on $C$. Tensoring $pr^*L$ with the short exact sequence (\ref{eq:10}), we obtain a short exact sequence 
$$
0\longrightarrow \sO_{C_{p}}\boxtimes L\longrightarrow (\sO_{C_{p}}\boxtimes L)(U_p)\longrightarrow (\sO_{C_{p-1}}\boxtimes L\otimes \omega^{-1}_C)(U_{p-1})\longrightarrow 0
$$
on $U_{p+1}$. Pushing it down to $C_p$ by the projection map $\pi_{p}:U_{p+1}\rightarrow C_p$ yields a connecting map $\delta$ in the associated long exact sequence
$$0\longrightarrow H^0(L)\otimes \sO_{C_p}\longrightarrow \pi_{p,*}((\sO_{C_{p}}\boxtimes L)(U_p))\longrightarrow \sigma_{p,*}((\sO_{C_{p-1}}\boxtimes L\otimes \omega^{-1}_C)(U_{p-1}))\stackrel{\delta}{\longrightarrow} \cdots \hspace{2cm}$$
$$\hspace{7cm} \cdots \stackrel{\delta}{\longrightarrow}H^1(L)\otimes \sO_{C_p}\longrightarrow R^1\pi_{p,*}((\sO_{C_{p}}\boxtimes L)(U_p))\longrightarrow 0,$$
where $\sigma_p$ is the restriction of $\pi_p$ onto the divisor $U_p$. To understand the connecting map $\delta$, we consider its dual map $\delta^\vee$ by applying $\sHom(-, \sO_{C_{p+1}})$. It is easy to calculate that 
$$\big(\sigma_{p,*}((\sO_{C_{p-1}}\boxtimes L\otimes \omega^{-1}_C)(U_{p-1}))\big)^\vee=\sigma_{p,*}(\sO_{C_{p-1}}\boxtimes L^{-1}\otimes \omega_C)=E_{p,L^{-1}\otimes\omega_C}.$$
Then the dual map $\delta^\vee$ turns out to be the evaluation map
$$H^0(L^{-1}\otimes \omega_C)\otimes \sO_{C_p}\stackrel{\delta^\vee}{\longrightarrow} E_{p,L^{-1}\otimes \omega_C}.$$
We shall only need the special case that $L=\omega_C$. In this case, the map $\delta^\vee$ splits $E_{p,\sO_C}$ by the trace map. As a consequence of the splitting, we have 
$$
(\sigma_{p,*}\sO_{U_p})^\vee \cong\sigma_{p,*}(\sO_{U_{p}}(U_{p-1}))\cong \sO_{C_p}\oplus \sK_p,
$$
where the direct summand $\sK_p$ is the kernel sheaf of the connecting map $\delta$  fitting into a short exact sequence
$$
0\longrightarrow H^0(\omega_C)\otimes \sO_{C_p}\longrightarrow \pi_{p,*}((\sO_{C_{p}}\boxtimes \omega_C)(U_p))\longrightarrow \sK_p\longrightarrow 0.
$$

\begin{theorem}\label{thm:01} Let $B$ be a $p$-very ample line bundle on $C$. Consider a line bundle $L:=B\otimes \omega_C$. Suppose that $h^0(B)\geq p+k+2$ for $k\geq 1$. Then 
	\begin{equation}\label{eq:01}
			H^i(U_{p+1}, \sigma^*_{p+1}(\wedge^kM_{p+1,B})\otimes (N_{p,L}\boxtimes L))=0\quad \text{ for }i>0.		
	\end{equation}
As a consequence, one has 
$$H^i(C_{p+1},\wedge^kM_{p+1,B}\otimes N_{p+1,L})=0 \quad \text{ for }i>0.$$
\end{theorem}
\begin{proof} First observe that by \cite[Lemma 3.5]{ENP}, $\sO_{C_{p+1}}(-\delta_{p+1})$ is a direct summand of the vector bundle $\sigma_{p+1,*}(\sO_{C_p}(-\delta_p)\boxtimes\sO_C)$. Thus the bundle
	$$\sigma_{p+1,*}( \sigma_{p+1}^*(\wedge^kM_{p+1,B})\otimes (N_{p,L}\boxtimes L))\cong\wedge^kM_{p+1,B}\otimes S_{p+1,L}\otimes \sigma_{p+1,*}(\sO_{C_p}(-\delta_p)\boxtimes\sO_C)$$ 
	 contains $\wedge^kM_{p+1,B}\otimes N_{p+1,L}$ as a direct summand. Since $\sigma_{p+1}$ is a finite map, the second vanishing statement in the theorem would follow from the first one. Thus in the sequel, it suffices to show the first vanishing statement (\ref{eq:01}).  

To this end, we use the short exact sequence in Lemma \ref{lm:01} to yield a locally free resolution of $\sigma^*_{p+1}(\wedge^kM_{p+1,B})$ as follows:
$$
\cdots\longrightarrow (\wedge^{k+2}M_{p,B}\boxtimes B^{-2})(2U_p)\longrightarrow (\wedge^{k+1}M_{p,B}\boxtimes B^{-1})(U_p)\longrightarrow \sigma^*_{p+1}(\wedge^kM_{p+1,B})\longrightarrow 0.
$$
Tensoring it with $N_{p,L}\boxtimes L$ gives rise to a resolution 
$$
\cdots\longrightarrow (\wedge^{k+2}M_{p,B}\otimes N_{p,L})\boxtimes (L\otimes B^{-2})(2U_p)\longrightarrow (\wedge^{k+1}M_{p,B}\otimes N_{p,L})\boxtimes (L\otimes B^{-1})(U_p)\longrightarrow \cdots\hspace{2cm}$$
	$$\hspace{8cm} \cdots\longrightarrow\sigma^*_{p+1}(\wedge^kM_{p+1,B})\otimes (N_{p,L}\boxtimes L)\longrightarrow 0.
$$
We make the following claim:

\begin{claim} One has
	$$R^tpr_{*}\Big((\wedge^{k+j}M_{p,B}\otimes N_{p,L})\boxtimes (L\otimes B^{-j})(jU_p)\Big)=0~~ \text{ for } t\geq 1, j\geq 2,$$
	where $pr \colon U_{p+1}\rightarrow C$ is the projection map.
\end{claim}
{\em Proof of Claim}. For a point $x\in C$, the restriction of the sheaf $(\wedge^{k+j}M_{p,B}\otimes N_{p,L})\boxtimes (L\otimes B^{-j})(jU_p)$ onto the fiber $pr^{-1}(x)\cong C_p$ equals $\wedge^{k+j}M_{p,B}\otimes N_{p,L(jx)}$, and $H^t(\wedge^{k+j}M_{p,B}\otimes N_{p,L(jx)}) = 0$ for $t > 0$ by Proposition \ref{lm:02}. Thus the claimed vanishing holds by base change. 

\medskip

By the claim above and using Larry spectral sequence 
$$H^s(R^tpr_{*}((\wedge^{k+j}M_{p,B}\otimes N_{p,L})\boxtimes (L\otimes B^{-j})(jU_p)))\Rightarrow H^{s+t}(U_{p+1},(\wedge^{k+j}M_{p,B}\otimes N_{p,L})\boxtimes (L\otimes B^{-j})(jU_p)),$$
we see that 
$$H^i(U_{p+1},(\wedge^{k+j}M_{p,B}\otimes N_{p,L})\boxtimes (L\otimes B^{-j})(jU_p))=0, \text{ for }i\geq 2, j\geq 2.$$
Thus  chasing through the resolution of  $\sigma^*_{p+1}(\wedge^{k}M_{p+1,B})\otimes (N_{p,L}\boxtimes L)$, in order to prove the vanishing (\ref{eq:01}),  the only left thing is to show the case when $j=1$, i.e.,  to show 
\begin{equation}\label{eq:03}
	H^i(U_{p+1},(\wedge^{k+1}M_{p,B}\otimes N_{p,L})\boxtimes \omega_C (U_p))=0,
\end{equation}
where we use the fact $L\otimes B^{-1}\cong \omega_C$.
To do this, we tensor $(\wedge^{k+1}M_{p,B}\otimes N_{p,L})\boxtimes\omega_C$ with the short exact sequence (\ref{eq:10}).  Pushing down the resulting sequence to $C_p$ by the projection map $\pi_p:U_{p+1}\rightarrow C_{p}$, we obtain  a long exact sequence
$$ 0\longrightarrow \wedge^{k+1}M_{p,B}\otimes N_{p,L}\otimes H^0(\omega_C)\longrightarrow \wedge^{k+1}M_{p,B}\otimes N_{p,L}\otimes \pi_{p,*}(\sO_{C_p}\boxtimes \omega_C(U_p))\longrightarrow\ldots \hspace{2cm}$$
$$\hspace{2.5cm}\ldots\longrightarrow\wedge^{k+1}M_{p,B}\otimes N_{p,L}\otimes \sigma_{p,*}(\sO_{U_p}(U_{p-1}))\stackrel{\delta}{\longrightarrow}  \wedge^{k+1}M_{p,B}\otimes N_{p,L}\longrightarrow  \cdots$$
$$\hspace{8cm}\cdots\stackrel{\phantom{b} }{\longrightarrow} R^1\pi_{p,*}(\wedge^{k+1}M_{p,B}\otimes N_{p,L}\boxtimes\omega_C(U_p))\longrightarrow 0.$$
As in the discussion located before the theorem, the connecting map $\delta$ splits.  This means that $R^1\pi_{p,*}(\wedge^{k+1}M_{p,B}\otimes N_{p,L}\boxtimes\omega_C(U_p))=0$ and $ \wedge^{k+1}M_{p,B}\otimes N_{p,L}$ is a direct summand of $\wedge^{k+1}M_{p,B}\otimes N_{p,L}\otimes \sigma_{p,*}(\sO_{U_p}(U_{p-1}))$. Thus we reduce the vanishing (\ref{eq:03}) to showing the vanishing
\begin{equation}\label{eq:04}
	H^i(C_p,\wedge^{k+1}M_{p,B}\otimes N_{p,L}\otimes \sigma_{p,*}(\sO_{U_p}(U_{p-1})))=0.
\end{equation} 
Observe that $$N_{p,L}\otimes \sigma_{p,*}(\sO_{U_p}(U_{p-1}))=\sigma_{p,*}(N_{p-1,L}\boxtimes L).$$ By projection formula, the vanishing (\ref{eq:04}) would follow from the following vanishing 
\begin{equation}
	H^i(U_{p},\sigma^*_p(\wedge^{k+1}M_{p,B})\otimes (N_{p-1,L}\boxtimes L))=0.
\end{equation}
Repeating this argument and noticing that $B$  is $(p-1)$-very ample with $h^0(B)\geq (p-1)+(k+1)+2$, we finally reduce the problem to showing the vanishing 
$$H^i(C,\wedge^{k+p}M_{B}\otimes L)=0,$$
Here we write $M_B=M_{1,B}$ for simplicity.
The only nontrivial case is when $i=1$. Write $b=\rank M_B$ and notice that $\det M^\vee_B=B$. By Serre duality,
$$H^1(C,\wedge^{k+p}M_B\otimes L)\cong H^0(C, \omega_C\otimes \det M^\vee_B\otimes \wedge^{b-1-k-p}M_B\otimes L^{-1})^\vee=H^0(C, \wedge^{b-1-k-p}M_B)^\vee.$$
Now as $\wedge^{b-1-k-p}M_B$ is a direct summand of $\otimes^{b-1-k-p}M_B$ and the latter has no global sections, we conclude $H^1(C,\wedge^{k+p}M_B\otimes L)=0$ as desired. This completes the proof.
\end{proof}

\begin{proposition}\label{p:12}  Let $B$ be a $p$-very ample line bundle on a curve $C$. Consider a line bundle $L=B\otimes \omega_C$. 
	\begin{enumerate}
		\item 	If $h^0(B)= p+k+1$ for $k\geq 1$. Then 
		$$H^i(C_{p+1}, \wedge^kM_{p+1,B}\otimes N_{p+1,L})=H^i(C_{p+1},S_{p+1,\omega_C})=S^{p+1-i}H^0(\omega_C).$$
		\item If $h^0(B)=p+k$ for $k\geq 1$, then $\wedge^kM_{p+1,B}=0$, and therefore
				$$H^i(C_{p+1}, \wedge^kM_{p+1,B}\otimes N_{p+1,L})=0.$$
		
	\end{enumerate}

\end{proposition}
\begin{proof} For (1), since $M_{p+1,B}$ has rank $k$ and $\wedge^k M_{p+1,B}\cong N^{-1}_{p+1,B}\cong S_{p+1,B^{-1}}(\delta_{p+1})$, we compute $$\wedge^kM_{p+1,B}\otimes N_{p+1,L}\cong S_{p+1,L\otimes B^{-1}}\cong S_{p+1,\omega_C}.$$
	The result then follows from \cite[Lemma 3.7]{ENP}. For	(2), since $M_{p+1,B}$ has rank $k-1$, the result follows immediately. 
\end{proof}

We will only need Theorem \ref{thm:01} and Proposition \ref{p:12} for the case $k=1$. In the following proposition, we classify when a $p$-very ample line bundle $B$ can have $h^0(B) \leq p+2$. 

\begin{proposition}\label{p:11} Let $B$ be a $p$-very ample line bundle on $C$, and $p\geq 0$. 
	\begin{enumerate}
		\item $h^0(B)=p+1$ if and only if either $p=0$ and $B=\sO_C$ or $p\geq1$, $C=\nP^1$ and $B=\sO_{\nP^1}(p)$.
		\item $h^0(B)=p+2$ if and only if one of the following cases holds.
		\begin{enumerate}
			\item [(i)] $g=0$, $p\geq 0$ and $B=\sO_{\nP^1}(p+1)$.
			\item [(ii)] $g=1$, $p\geq 0$ and $\deg B=p+2$.
			\item [(iii)] $g\geq 2$, either $p=0$ and $B$ is a base point free pencil, or $p=1$ and $C\subseteq \nP(H^0(B))$ is a plane curve of degree $\geq 4$.
		\end{enumerate}
	\end{enumerate}
\end{proposition}
\begin{proof} (1) If $p=0$, then $B$ is a globally generated line bundle with $H^0(C,B)=1$. Then $B=\sO_C$ since the only section of $B$ is nowhere vanishing.  Assume $p\geq 1$, so $B$ is very ample and gives an embedding of $C$ into the space $\nP^p=\nP(H^0(B))$. As $B$ is $p$-very ample, any $p+1$ points of $C$ will span the whole space $\nP^p$, which means that the degree of $C$ would be smaller than $p$. But $C$ is also nondegenerate in $\nP^p$ and thus has degree $\geq p$. Hence $C$ has degree exactly $p$, and therefore, it is  a rational normal curve. 

\noindent (2) Since (i) and (ii) are obvious, we only need to prove (iii). If $p=0$, then $B$ is a base point free pencil. Assume that $p\geq 2$. Take $p-1$ points $x_1,\ldots, x_{p-1}$ of $C$, and put $D:=x_1+\cdots+x_{p-1}$.  Since $B$ is a $p$-very ample, we see that $B(-D)$ is very ample with $h^0(B(-D))=3$ and $h^1(B(-D+x_1))=h^1(B(-D))=h^1(B)$. This means $C$ is a plane curve of some degree $d\geq 4$ embedded by $B(-D)$ into $\nP^2$, and thus, the canonical line bundle $\omega_C$ has the form $\omega_C=(B(-D))^{d-3}$ by the adjunction formula. By duality, the equality $h^1(B(-D+x_1))=h^1(B(-D))$ is the same as the equality $h^0((B(-D))^{d-4}(-x_1))=h^0((B(-D))^{d-4})$, which is impossible because $B(-D)$ is very ample. Thus we conclude $p=1$  and $C\subseteq \nP(H^0(B))$ is a plane curve of degree $\geq 4$.	
\end{proof}

Recall that the gonality of $C$ captures the positivity of the canonical line bundle $\omega_C$. More precisely, $\gon(C) \geq p+2$ if and only if $\omega_C$ is $p$-very ample. In particular,
$$
\gon(C) = \max\{p+2 \mid \text{$\omega_C$ is $p$-very ample}\}.
$$
We can compare the gonality with the genus. The following proposition may be well-known.

\begin{corollary}\label{p:15} Assume that $g\geq 2$. Then $g\geq \gon(C)$, and the equality holds if and only if either $g=2$ or $C$ is a plane quartic curve.
\end{corollary}

\begin{proof} 
Since $g\geq 2$, it follows that $\gon(C)\geq 2$. Write $\gon(C)=p+2$. Then $\omega_C$ is $p$-very ample. Applying Proposition \ref{p:11} to the case $B=\omega_C$, we see that $g\geq p+2$ and the equality holds if either $g=2$ (i.e., $g=\gon(C)=2$), or $C$ is a  plane curve of $g=3$ which is a plane quartic curve (i.e., $g=\gon(C)=3$).
\end{proof}

\begin{proof}[Proof of Theorem \ref{main02}] In (1) and (2), $B$ is ample and thus $h^1(L)=0$. This implies  $h^1(C_{p+1}, N_{p+1,L})=0$ and thus \cite[Lemma 1.1]{Ein:Gonality} yields
$$
K_{p,1}(C,B;L)=H^1(C_{p+1},M_{p+1,B}\otimes N_{p+1,L}).
$$
So the assertion (1) follows from Theorem \ref{thm:01} by taking $k=1$, and the assertion (2) follows from Proposition \ref{p:12} by taking $k=1$. For the assertion (3), if $p=0$, then $B=\sO_C$ and then $K_{0,1}(C;\omega_C)=0$ by definition of Koszul cohomology group. If $p\geq 1$, then by Proposition \ref{p:11}, $C=\nP^1$ and $B=\sO_{\nP^1}(p)$. Then $K_{p,1}(\nP^1,\sO_{\nP^1}(p);\sO_{\nP^1}(p-2))=0$ by a direct computation. 
\end{proof}

\begin{corollary}\label{cor:main01}
Assume that $g \geq 2$. Let $B$ be a $p$-very ample line bundle on $C$, and $L$ be a line bundle on $C$. Suppose that $\deg (L\otimes B^{-1}) \geq 2g-2$. Then one has 
	$$K_{p,1}(C,B;L)=0$$
unless $L=B\otimes \omega_C$ and either $(1)$ $p=0$ and $B$ is a base point free pencil, or $(2)$ $p=1$ and $C\subseteq\nP(H^0(B))$ is a plane curve.  In the exceptional cases, $K_{p,1}(C,B;L)\neq 0$ but $K_{p-1, 1}(C, B; L) = 0$.
\end{corollary}

\begin{proof} 
If $L\otimes B^{-1}\neq \omega_C$, then $h^1(L \otimes B^{-1})=0$ so that one can use Rathmann's theorem \cite[Theorem 1.1]{RATH} to get the desired result. Assume that $L\otimes B^{-1}=\omega_C$. By Theorem \ref{main02}, $K_{p,1}(C,B;L)=0$ if $h^0(B) \neq p+2$, and $K_{p,1}(C, B; L) \neq 0$ if $h^0(B) = p+2$. In the latter case, $K_{p-1,1}(C, B; L) = 0$ by Theorem \ref{main02} since $B$ is $(p-1)$-very ample and $h^0(B) = (p-1)+3$. However, if $h^0(B) = p+2$, then Proposition \ref{p:11} shows that either $(1)$ $p=0$ and $B$ is a base point free pencil, or $(2)$ $p=1$ and $C\subseteq\nP(H^0(B))$ is a plane curve.
\end{proof}

\begin{proof}[Proof of Theorem \ref{main01}] 
By Green--Lazarsfeld's nonvanishing theorem \cite[Appendix]{Green:KoszulI} and the duality theorem \cite[Theorem 2.c.6]{Green:KoszulI}, we only need to know when $K_{\gon(C)-2,1}(C,\omega_C;L)=0$ vanishes. As $\omega_C$ is $(\gon(C)-2)$-very ample, the theorem follows from Corollary \ref{cor:main01}.
\end{proof}

\bibliographystyle{alpha}

\begin{thebibliography}{ENP20}
	
	\bibitem[AN10]{AN}
	Marian Aprodu and Jan Nagel.
	\newblock Koszul Cohomology and Algebraic Geometry.
	\newblock University Lecture Series, 52, Amer. Math. Soc., Providence, RI, 2010.
	
	\bibitem[EL15]{Ein:Gonality}
	Lawrence Ein and Robert Lazarsfeld.
	\newblock The gonality conjecture on syzygies of algebraic curves of large
	degree.
	\newblock {\em Publ. Math. Inst. Hautes \'Etudes Sci.}, 122:301--313, 2015.
	
	\bibitem[ENP20]{ENP}
	Lawrence Ein, Wenbo Niu, and Jinhyung Park.
	\newblock Singularities and syzygies of secant varieties of nonsingular
	projective curves.
	\newblock {\em Invent. Math.}, 222(2):615--665, 2020.
	
	\bibitem[Gr84]{Green:KoszulI}
	Mark Green.
	\newblock Koszul cohomology and the geometry of projective varieties.
	\newblock {\em J. Differential Geom.}, 19(1):125--171, 1984.
	
	\bibitem[GL86]{Lazarsfeld:ProjNormCurve}
	Mark Green and Robert Lazarsfeld.
	\newblock On the projective normality of complete linear series on an algebraic
	curve.
	\newblock {\em Invent. Math.}, 83(1):73--90, 1986.
	
	\bibitem[GL88]{Lazarsfeld:SyzFiniteSet}
	Mark Green and Robert Lazarsfeld.
	\newblock Some results on the syzygies of finite sets and algebraic curves.
	\newblock {\em Compositio Math.}, 67(3):301--314, 1988.
	
	
	\bibitem[La04]{Pos}
	Robert Lazarsfeld.
	\newblock Positivity in algebraic geometry. I. Classical setting: line bundles and
	linear series. II. Positivity for vector bundles, and multiplier ideals.
	\newblock {\em Springer-Verlag}, 2004
	
		
	\bibitem[Ra16]{RATH}
	Juergen Rathmann.
	\newblock An effective bound for the gonality conjecture.
	
	
	\bibitem[Voi1]{Voisin1}
	Claire Voisin.
	\newblock \textit{Green’s generic syzygy conjecture for curves of even genus lying on a K3 surface}.
	\newblock J. Eur. Math. Soc. (JEMS) \textbf{4} (2002), 363--404.
	
	\bibitem[Voi2]{Voisin2}
	Claire Voisin.
	\newblock \textit{Green’s canonical syzygy conjecture for generic curves of odd genus}.
	\newblock Compositio Math. \textbf{141} (2005), 1163--1190.
	
\end{thebibliography}

\end{document}